\documentclass[a4paper,12pt]{amsart}
\usepackage{amsmath} 
\usepackage{amssymb} 
\usepackage{amsthm} 
\usepackage[import]{xy}

	\usepackage{graphicx}
 \usepackage[margin=15pt,format=hang,font=small,labelfont=bf,justification=centering]{caption} 

\newtheorem{theorem}{Theorem}[section] 
\newtheorem{lemma}[theorem]{Lemma}

\theoremstyle{definition}

\newtheorem{definition}[theorem]{Definition}

\newcommand{\R}{\ensuremath{\mathbb{R}}} 
 
\newcommand{\integer}{\ensuremath{\mathbb{Z}}}

\newcommand{\T}{\ensuremath{\mathcal{T}}}

\newcommand{\scirc}{\raise1pt\hbox{$\,\scriptstyle\circ\,$}}

\DeclareMathOperator{\inn}{int}
\DeclareMathOperator{\aff}{aff}

\newsavebox{\matrixbox}

\begin{document}

\title[Parallelogram tilings, Worms and Finite Orientations]{Parallelogram tilings, Worms and Finite Orientations}

%[Parallelogram tilings and infinite rotations]

\author[D.~Frettl\"oh]{Dirk Frettl\"oh}
\address{Institut f\"ur Mathematik, FU Berlin}
\email{frettloe@mi.fu-berlin.de}
\urladdr{http://page.mi.fu-berlin.de/frettloe/}

\author[E.~Harriss]{Edmund Harriss}
\address{Department of Mathematical Sciences, University of Arkansas}
\email{edmund.harriss@mathematicians.org.uk}
\urladdr{http://www.mathematicians.org.uk/eoh}

\date{\today}

\maketitle

\begin{abstract}
This paper studies properties of tilings of the plane by  
parallelograms. In particular it is established that 
in parallelogram tilings using a finite number of shapes all tiles 
occur in only finitely many orientations.
\end{abstract}

\section{Introduction}

Parallelogram tilings, in particular rhomb tilings, are a fundamental family in the mathematical study of tilings, in both aperiodic order with its links to quasicrystals and in the study of random tilings. In addition they provide visual models for a wide range of mathematical structures. 

The most famous aperiodic parallelogram tiling is certainly the Penrose tiling~\cite{Penrose:PCNTP, Gardner:ENTTE, Harriss:TE} in the version with rhombic tiles. Many other aperiodic tilings use parallelograms as tiles, including the Ammann-Beenker tilings~\cite{Ammann:AT,Beenker:ATONP}, or more generally, all canonical projection tilings~\cite{Forrest:COCPT,Harriss:CSTOA}. Further examples arise in the study of geometric representations of substitutive dynamical systems,see for instance~\cite{Fogg:SIDAC}. Such models are of particular importance in the study of mathematical quasicrystals~\cite{Shechtman:MPWLR} and systems with long range aperiodic order~\cite{Baake:DIMQ, Moody:TMOLR}. 

In random tilings parallelogram tilings are also of importance. Usually finite tilings are considered, lozenge tilings of polygons~\cite{Kenyon:LSCBE,Bodini:DRT}, or parallelotope tilings of zonotopes~\cite{Destainville:RTHSIMT,Chavanon:RTDSS}, but infinite random parallelogram tilings are also studied~\cite{Bodini:CFRTWP}. In this paper we focus on infinite tilings of the plane, but some results are also valid and relevant for finite tilings, for instance Lemma~\ref{lem:crossing_lemma} and Lemma~\ref{lem:travel}. 

Parallelogram tilings also turn up in the study of pseudo-line arrangements in combinatorics (for example~\cite{Agarwal:PADAA,Goodman:PCBGS}), as a visualisation of reduced words in Coxeter groups~\cite{Elnitsky:RTPCRWCG} and, through random tilings and Dimer models, they link to a wide variety of physical models~\cite{Stienstra:HSVQDDD}. The use of parallelogram tilings in a wide variety of mathematics is discussed further on Math Overflow~\cite{Speyer:RTWMTTD}.

Interestingly, there are aperiodic tilings of the plane with polygonal 
tiles where the tiles occur in an infinite number of orientations: The 
Pinwheel Tiling~\cite{Radin:TPTOT} is an example of a tiling using only one kind
of tile (a triangle with edge lengths $1,2$ and $\sqrt{5}$) where the tiles 
occur in infinitely many orientations throughout the tiling. A 
comprehensive study of this phenomenon is~\cite{Frettloh:STWSCS}. There are also 
tilings using two kinds of quadrilaterals in which the tiles occur in 
infinitely many orientations~\cite{Baake:RAPSFWAPDPP}, see for instance the Kite Domino 
Tiling in the online resource~\cite{Harriss:TE}. (Click ``Kite-Domino'' on the 
main page for information about the Kite Domino Tiling, and then 
``Infinite rotations'' for a list of other examples.) 
Our main result states that such a behaviour is impossible if all tiles 
are parallelograms. 

By definition, a parallelogram has two pairs of parallel edges. In a 
parallelogram tiling therefore there are natural lines of tiles linked 
each sharing a parallel edge with the next. These lines are called worms 
and were used by Conway in studying the Penrose tilings~\cite{Grunbaum:TandP}. We will 
follow Conway to call these lines \emph{worms}. 
Figure~\ref{fig:Tile_between_two_worms} shows three (finite pieces of) such worms. 
In this paper we consider the structure of these worms and use them to 
prove results on parallelogram tilings. In particular, our main result 
states that no parallelogram tiling using finitely many different 
parallelograms can have tiles in infinitely many orientations. 

This paper is organised as follows: The first section defines 
the necessary terms and notations. The next section is dedicated to the 
proof of our main result Theorem~\ref{thm:infinite_orientations}. It 
contains several results which may be of interest on their own in the 
study of parallelogram tilings, for instance the Crossing Lemma \ref{lem:crossing_lemma} 
or the Travel Lemma~\ref{lem:travel}. The last 
section contains some additional remarks. 

A tiling $\T$ of the plane $\R^2$ is a packing of $\R^2$ which is also a 
covering of $\R^2$. Usually, the tiles are nice compact sets, and the 
tiling is a countable collection of tiles $T_i$: $\T = \{ T_1, T_2, \ldots\}$. 
The covering condition can then be formulated as $\bigcup_i T_i = \R^2$, 
and the packing condition as $\inn(T_i) \cap \inn(T_j) = \varnothing$ 
for $i \ne j$, where $\inn(A)$ denotes the interior of $A$. 

The set of congruence classes of tiles in a tiling $\T$ is called 
{\em protoset} of $\T$. Its elements are called {\em prototiles} of $\T$. 
A tiling is called {\em locally finite}, if each ball $B \subset \R^2$ 
intersects only a finite number of tiles of $\T$. A tiling $\T$ is 
called {\em vertex-to-vertex}, if all prototiles are polygons, and 
for all tiles in $\T$ holds: If $x \in \R^2$ is vertex of some tile 
$T \in \T$, then $x \in T' \in \T$ implies $x$ is a vertex of $T'$. 

{\bf General Assumption:} Throughout the paper, only locally finite 
vertex-to-vertex tilings  are considered. Moreover, the the prototiles 
will always be parallelograms. In particular, all tiles have nonempty 
interior. 

\begin{definition}
\label{def:worm}
Let $T$ be a tile in a parallelogram tiling, and let $e$ be one of its edges. There is a unique tile $T_1$ sharing the edge $e$ with $T$. Let $e_1 \ne e$ denote the edge of $T_1$ parallel to $e$. There is a unique tile $T_2$ sharing this edge with $T_1$. Repeating this yields a (one-sided) infinite sequence of tiles $T,T_1, T_2, \ldots$ uniquely given by $T$ and $e$. Repeating this on the edge of $T$ which is opposite of $e$ yields a further  (one-sided) infinite sequence of tiles $T_{-1}, T_{-2}, \ldots$. In this way, $T$ and $e$ define a unique biinfinite sequence $W:=\ldots T_{-2}, T_{-1}, T, T_1, T_2, \ldots$. We will call this sequence the {\em worm given by $T$ and $e$}, and any such sequence just a {\em worm}. The intersection of two worms $W, W'$ is defined as $W \cap W' = \{ T \, | \, T \in W, T \in W' \}$. 
\end{definition}

Worms are also called {\em de Bruijn lines} \cite{Destainville:RTHSIMT}
or {\em ribbons} \cite{Bodini:DRT}. Worms correspond 
to pseudo-line configurations~\cite{Agarwal:PADAA,Goodman:PCBGS}, 
and to grid lines used in the construction of canonical projection tilings 
by the multigrid method~\cite{Gahler:EOTGG}. 

\begin{definition} % Worm Families
\label{def:worm_species}
Let $e$ be some edge in a parallelogram tiling. The set of all worms 
which are defined by edges parallel to $e$ is called the 
{\em worm family (given by $e$)}. 
\end{definition}

\section{Proof of the Main Result} % (fold)

\label{sec:basic_results}

In this section we collect some basic results on worms, which will be 
used in the next section. The first result states that a worm cannot 
bend ``too much''. In order to state this precisely, we define the 
open cone $C(x,\alpha)$ with axis $x \in \R^2$ and angle $\alpha \in 
[0,\pi[$ by $C(x,\alpha) := \{ z \in \R^2 \, | \, |\angle(x,z)| < \alpha \}$, 
where $\angle(x,z) = \arccos \frac{ \langle x,z \rangle}{\|x\| \|z\|}$ 
denotes the angle between $x$ and $z$.

%\begin{figure}[htb]
%\begin{center}
%\includegraphics[width=100mm]{nogo-cones}
%\end{center}
%\caption{The worm defined by the edge $e$ cannot intersect the open cones 
%$C_1$ or $C_2$.}
%\label{fig:Impossible_Cones}
%\end{figure}

\begin{figure}[htb]
	$$\begin{xy}
		\xyimport(80,110){\includegraphics{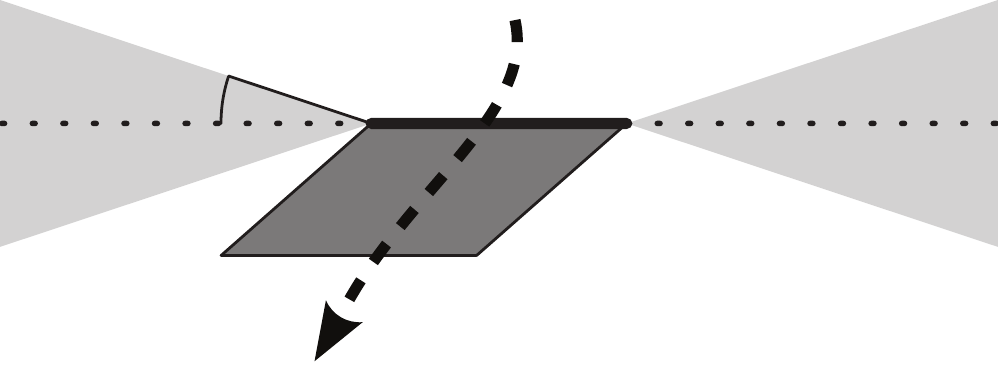}},
		(35,84)*!U\txt{e},
		(20,84)*!U\txt{$\alpha$},
		(12,82)*!U\txt{\Large{$C_1$}},
		(68,82)*!U\txt{\Large{$C_2$}}
	\end{xy}$$
	\caption{The worm defined by the edge $e$ cannot intersect the open cones $C_1$ or $C_2$.}
	\label{fig:Impossible_Cones}
\end{figure}

\begin{lemma}  \label{lem:cone_lemma}
Let $\T$ be some parallelogram tiling, with finite or infinite protoset. 
Denote the infimum of all interior angles of the prototiles by $\alpha$.
Let $W$ be a worm given by some tile $T$ and some edge $e$. Denote the 
two endpoints of $e$ by $e_1$ and $e_2$. If $\alpha>0$ then $W$ is 
contained in 
\[ \R^2 \setminus \big( C_1 \cup C_2 \big), \quad \mbox{where} \quad 
C_1:= e_1+C(e_1-e_2, \alpha), \; C_2 := e_2+C(e_2-e_1, \alpha). \]
If $\alpha=0$ then $W$ is contained in $e \cup \big( \R^2 \setminus 
\aff(e) \big)$, where $\aff(e)$ denotes the affine span of $e$.
\end{lemma}

\begin{proof}
By inspection of Figure~\ref{fig:Impossible_Cones} we observe 
that the worm defined by the edge $e$ cannot intersect the cones (shaded) 
$C_1$ and $C_2$, where $\alpha>0$ 
is the smallest interior angle of the prototiles. Using the 
prototile with the smallest angle $\alpha$ alone, the worm can line 
up along the boundary line of one of the cones at best.

If $\alpha=0$, the open cones are empty. Nevertheless, since any tile has 
nonempty interior, the particular tile 
in the worm $W$ containing $e$ as an edge has some positive interior 
angle. Then, by the same reasoning as above, $W$ cannot touch the 
line $\aff(e)$ apart from $e$. 
\end{proof}

The last lemma shows also that a worm can have no loop. We state this result as a an additional lemma. 

\begin{lemma} \label{lem:loop_lemma}
A worm has no loop.
\end{lemma}

\begin{lemma}[Crossing Lemma] \label{lem:crossing_lemma}
Two worms cross at most once. Worms in the same family never cross.
\end{lemma}

In other words: The intersection of any two worms is empty or consists 
of a single tile. The intersection of two different worms defined by 
parallel edges is empty.

\begin{figure}[htb]
		$$\begin{xy}
			\xyimport(80,110){\includegraphics{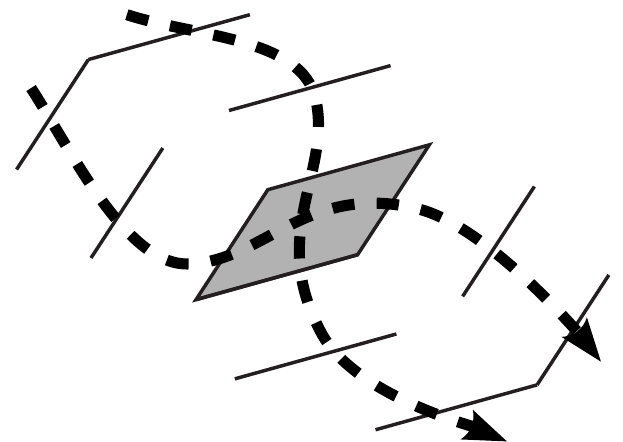}},
			(56,70)*!U\txt{e'},
			(43,42)*!U\txt{e},
		\end{xy}$$
	\caption{Two worms can cross at most once.}
	\label{fig:One_crossing}
\end{figure}

\begin{proof}
Consider two worms in distinct families, and put their defining edges 
together at a point (see Figure~\ref{fig:One_crossing}). Since the tiles
are parallelograms, and the interior angles of parallelograms are less 
than $\pi$,  the two worms can only cross in the direction where the 
edges form an angle less than $\pi$. After crossing this direction is 
reversed. Thus no further crossing is possible.

If the two worms are in the same family the crossing tile would have 
to have all four edges parallel, but no tile (with positive area) 
can have this property.
\end{proof}

\begin{lemma}[Travel Lemma] \label{lem:travel}
Let $\T$ be some parallelogram tiling, with finite or infinite protoset. 
Denote the infimum of all interior angles of the prototiles by $\alpha$.
If $\alpha > 0$, then any two tiles $S,T$ in $\T$ can be connected by 
a finite sequence 
of tiles which are contained in a finite number of worms. Moreover,
less than $\lceil 2 \pi / \alpha \rceil$ different worms are needed. 
(Here $\lceil x \rceil$ denotes the smallest $k \in \integer$ such 
that $k \ge x$.)
\end{lemma}

In plain words: If the prototiles are not arbitrarily thin, then we can 
always travel from tile $S$ to tile $T$ by walking on worms, with finitely 
``turns''; that is, with finitely many changes from one worm to another one.
The proof stresses the fact that tiles cannot be arbitrarily thin.

\begin{proof}
	\begin{figure}[htb]
			$$\begin{xy}
				\xyimport(80,110){\includegraphics{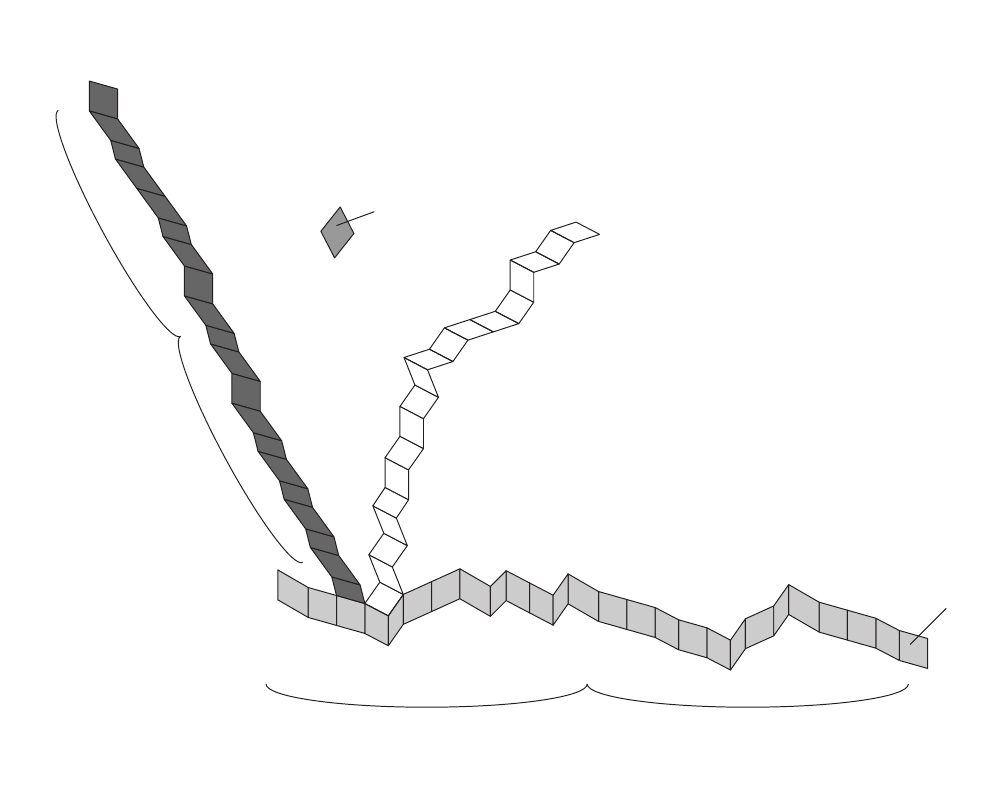}},
				(79,28)*!U\txt{$S$},
				(32,83.5)*!U\txt{$T$},
				(48,10)*!U\txt{$W_0$},
				(12,60)*!U\txt{$W_1$}
			\end{xy}$$
		\caption{Starting from a tile $S$ and worm $W_0$, there are a set of worms crossing $W_0$. A tile $T$ that does not lie on one of these worms must lie between two. We label the worm further from $S$ with $W_1$.}
		\label{fig:Tile_between_two_worms}
	\end{figure}
	
	We start from tile $S$, and choose $W_0$, one of the two worms 
passing through $S$. As, by Lemma~\ref{lem:loop_lemma} $W_0$ is not a loop, 
an infinite number of worms cross $W_0$.  The tile $T$ will either 
lie on one of these worms or between two of them. (The latter situation is 
indicated in Figure~\ref{fig:Tile_between_two_worms}.)
	
	If $T$ lies on one of the worms we are done. Otherwise there are 
two worms, defined by adjacent tiles in $W_0$, such that $T$ lies between 
these two worms. Let $W_1$ be the worm further from $S$ and $S_1$ be the 
tile given by the intersection of $W_0$ and $W_1$.  Again $T$ either lies 
on a worm crossing $W_1$ or between two worms. If $T$ does lie on such a 
worm we are done. Otherwise call the worm further from $S_1$ $W_2$. 
	
	% \begin{figure}[htb]
	% 		$$\begin{xy}
	% 			\xyimport(80,110){\includegraphics{No_way_out.pdf}},
	% 			(79,28)*!U\txt{$S$},
	% 			(32,83.5)*!U\txt{$T$},
	% 			(47,85)*!U\txt{lots of\\worms},
	% 			(40,55)*!U\txt{line that\\cannot\\be crossed},
				
	% 		\end{xy}$$
			
	% 	\caption{Eventually the sequence of worms will get trapped. It cannot go back on itself but ahead lie other worms that cannot be crossed.}
	% 	\label{fig:No_way_out}
	% \end{figure}

Using this method we obtain a sequence of worms $W_0, W_1, \ldots$.	
Note that we cannot visit infinitely many tiles of the same worm: 
In this case we would have found infinitely many worms passing between 
$S$ and $T$, which is impossible, since worms cannot be arbitrarily thin. 

Thus we obtain a sequence of finite pieces of worms by our method. We 
proceed to show that we have to arrive at $T$ in this way after finitely 
many turns. 

\begin{figure}[htb]
\begin{center}
\includegraphics[width=50mm]{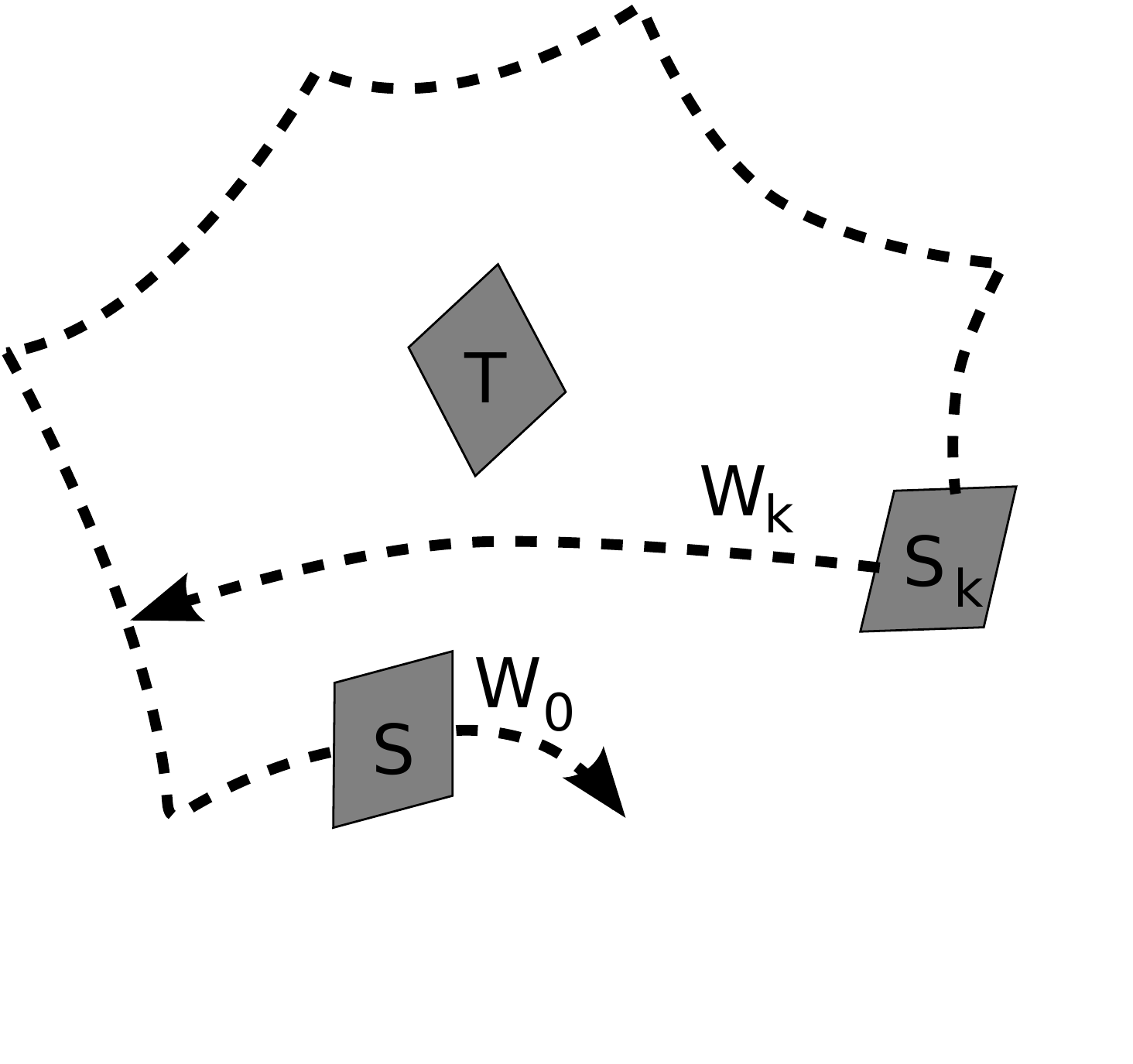}
\end{center}
\caption{Case 2: The constructed worm sequence encircles $T$.}
\label{fig:spiral}
\end{figure}

By assumption we have $\alpha > 0$. Thus after at most $k:=\lceil 2 
\pi / \alpha \rceil$ turns we completed a full turn by $2 \pi$ or more. 
Note that, by construction, the tile $T$ is always on the right hand 
side of the current worm, or always on the left hand side. Without 
loss of generality, let $T$ be always on the right hand side of
the current worm.
Now let us travel further on the current worm $W_k$, say. Either $S$ is 
contained in $W_k$ (Case 1), or on the left hand side of it (Case 2), or 
on the right hand side of it (Case 3). In the first two cases we have 
completed a loop (Case 2 is indicated in Figure~\ref{fig:spiral}).

\begin{figure}[bht]
\begin{center}
\includegraphics[width=60mm]{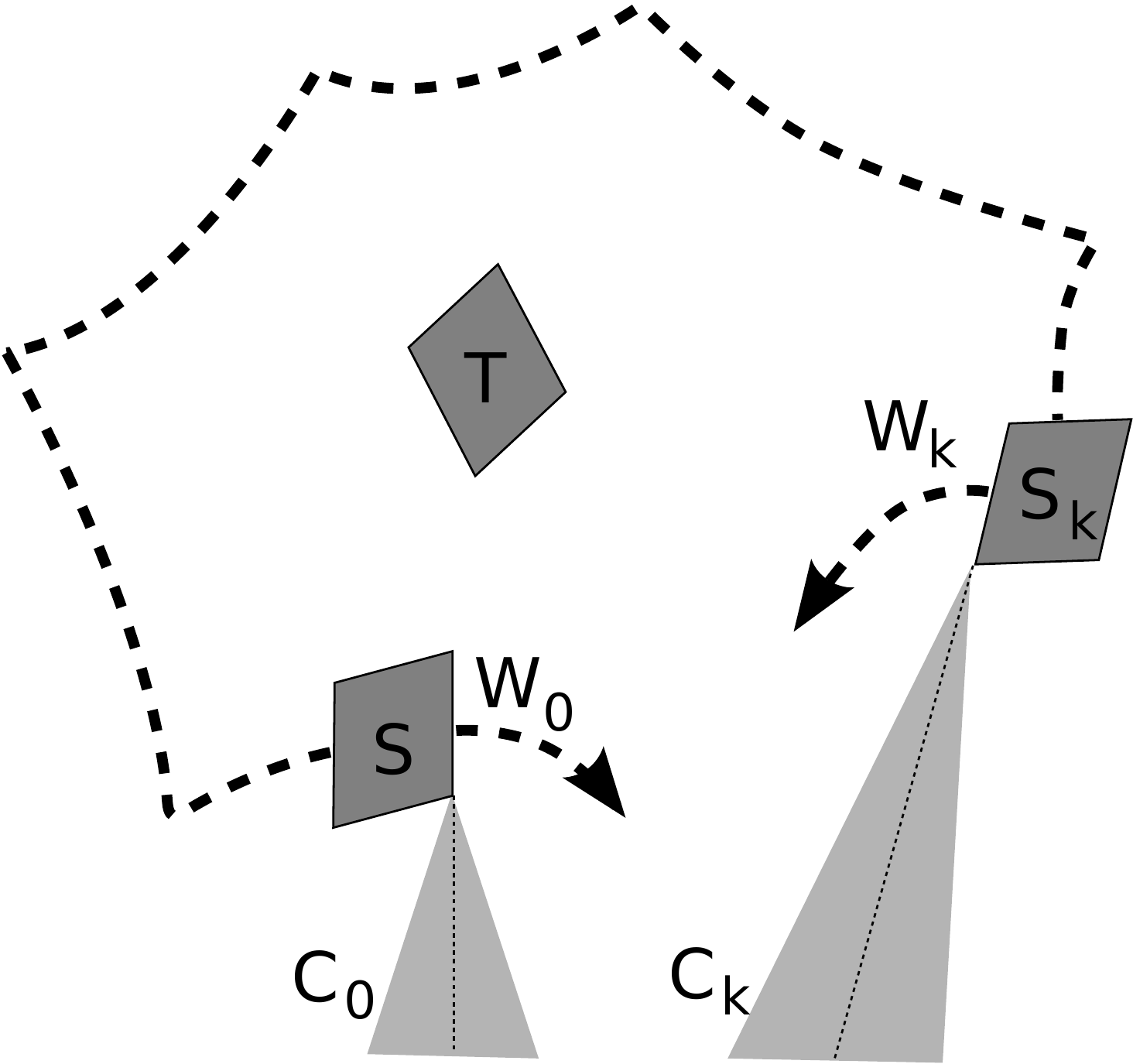}
\end{center}
\caption{Case 3: The constructed worm sequence does not entirely 
encircle $T$, but nevertheless $W_0$ and $W_k$ have to cross 
(since the cones $C_0$ and $C_k$ are tilted towards each other,
$W_0$ cannot intersect $C_0$, and $W_k$ cannot intersect $C_k$).  }
\label{fig:nowayout}
\end{figure}

Assume we have completed a loop. By construction, $T$ lies inside this 
loop. Since, by construction, our sequence of worms is not crossed by 
any worm defined by $T$, the two worms defined by $T$ must be contained 
entirely inside this loop. Since worms are infinite, and tiles are not 
arbitrarily small, this is impossible. 

It remains to consider Case 3: $S$ is on the right hand side of our 
current worm $W_k$. There are two subcases: Either $W_k$ reaches the 
first worm $W_0$, defined by $S$. In this case we have closed a loop 
again, yielding a contradiction. Or $W_k$ avoids the first worm $W_0$. 
Now we utilize the Cone Lemma~\ref{lem:cone_lemma} to show that 
this case is impossible, 
too. Let us denote the tile where we entered $W_k$ by $S_k$. By Lemma 
\ref{lem:cone_lemma}, the cone $C_0$ pointing outward our sequence 
(to the left) defined by $S$ and the edge defining $W_0$ cannot contain 
tiles of $W_0$. Similarly, the cone pointing outward defined 
by $S_k$ and the edge defining $W_k$ cannot contain tiles of $W_k$. 

Either $C_k$ intersects $W_0$, then $W_k$ intersects our worm sequence, 
closing a loop. Otherwise, since $k \ge 2\pi/\alpha$, 
the axis of $C_k$ is tilted towards $S$, thus $C_0$ and $C_k$ have 
non-empty intersection. (This situation is 
indicated in Figure~\ref{fig:nowayout}.) Neither $W_0$ nor $W_k$ can 
intersect $C_0 \cap C_k$. $W_0$ cannot intersect $C_0$, and $W_k$ 
cannot intersect $C_k$, thus $W_0$ and $W_k$ have to intersect each other. 
This closes a loop, which yields a contradiction.

Thus our constructed worm sequence cannot avoid $T$. Thus $T$ is reached 
from $S$ by using (finitely many) tiles contained in finitely many 
($k$ or less) worms. 
\end{proof}

\begin{theorem} % Infinite Orientations

\label{thm:infinite_orientations}

Let ${\mathcal T}$ be a parallelogram tiling with a finite protoset. 
Let $\alpha$ denote the minimal interior angle in the prototiles.
Each prototile occurs in a finite number of orientations in ${\mathcal T}$. 
Moreover, the number of orientations of each prototile is bounded by a 
common constant. 
\end{theorem}

\begin{proof}
This is now an immediate consequence of Lemma~\ref{lem:travel}: Fix 
some tile $S$ in the tiling. Any other tile $T$ of the same shape 
can be reached from $S$ by a sequence of worms with at most 
$k =  \lceil 2 \pi / \alpha \rceil$ turns. Let there be $m$ prototiles 
altogether. If we can connect $S$ with $T$ without any turn at all 
(i.e., $S$ and $T$ are contained in the same worm) there are $2$ 
possibilities how $T$ is oriented with respect to $S$. ($T$ can be 
a translated copy of $S$ or a reflected copy of $S$, yielding two 
different orientations unless $S$ and $T$ are squares.) If we can connect 
$S$ with $T$ using one turn, there are $2m$ possibilities  how $T$ is 
oriented with respect to $S$: $m$ possibilities arising from the 
choice of the prototile defining the turn, and two possibilities 
for any such choice (left-handed or right-handed). Analoguously,
for any turn there are $2m$ possibilities of changing the orientation.
Thus there are $(2m)^{\ell}$ possibilities if we need exactly 
$\ell$ turns. Altogether, this leaves 
\[ N:= 2 + 2m + (2 m)^2 + \cdots + (2 m)^k = 1+ \sum_{n=0}^k (2m)^k 
= \frac{1}{2m-1} \big( (2m)^{k+1} +2m -2 \big) \] 
as an upper bound for the number of different orientations of $T$ 
with respect to $S$. This is true for any tile $T$ congruent to $S$. 
Thus there are at most $N^2$ possibilities for the orientation of any two 
congruent tiles $T, T'$ to each other. 
\end{proof}

% \begin{figure}[htb]
% 	\centering
% 		\includegraphics{Infinite_orientations.pdf}
% 	\caption{This patch illustrates how, using an infinite set of prototiles, a tile (in this case a square) can lie in infinitely many orientations. 
% % The explicit construction is left as an exercise.
% }
% 	\label{fig:Infinite_orientations}
% \end{figure}

\section{Remarks}

\begin{figure}[bht]
\begin{center}
\includegraphics[width=110mm]{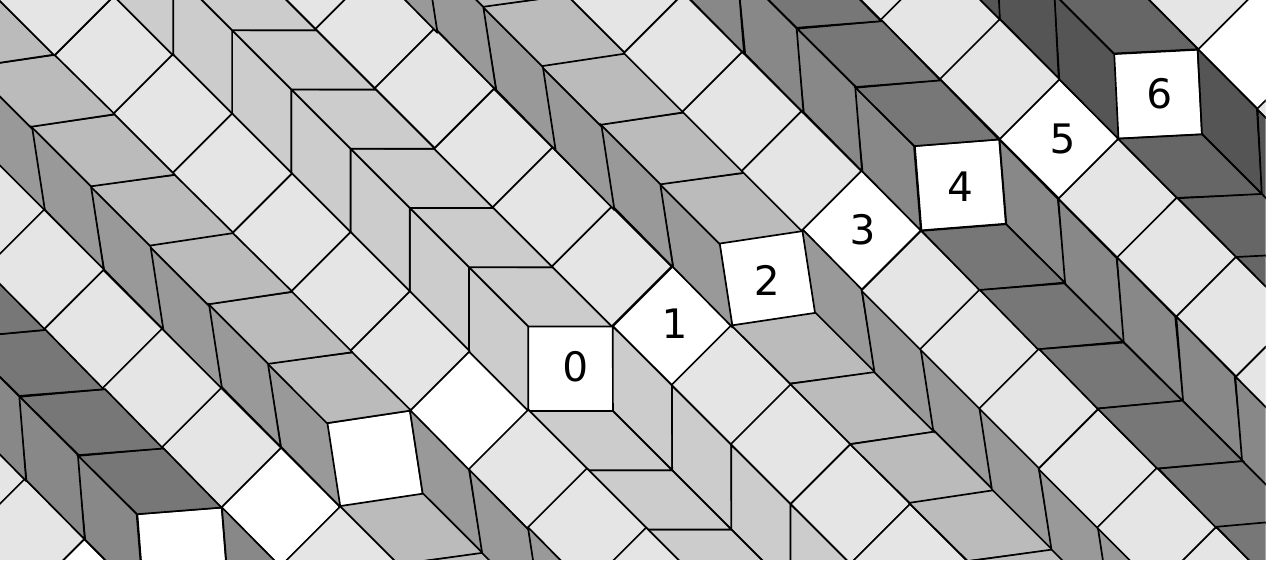}
\end{center}
\caption{This patch illustrates how a tile (in this case a square) 
can lie in infinitely many orientations, if we allow infinitely 
many prototiles.} 
\label{fig:Infinite_orientations}
\end{figure}

The tiling indicated in Figure~\ref{fig:Infinite_orientations} 
demonstrates that Theorem~\ref{thm:infinite_orientations} is sharp, 
in the following sense: 
it becomes false if we allow an infinite number of prototiles, even 
if the minimal interior angle $\alpha$ is bounded away from 0. 
In this example, the squares labelled $1,3,5,\ldots$ are rotated 
through $\frac{\pi}{4}$ with respect to the square labelled 0 
(compare Figure~\ref{fig:Infinite_orientations}). The squares with 
even label $i \ne 0$ are rotated through $\frac{\pi}{10i}$ with respect 
to the square labelled 0.
Thus these squares occur in infinitely many orientations in the tiling. 

Nevertheless, the Travel Lemma still holds in this example: The minimum
interior angle is $\alpha = \frac{\pi}{4}-\frac{\pi}{10}$ (occurring
in the thinner rhomb that touches square number 2). Thus any tile 
can be connected to the square 0 by a (finite) sequence of tiles which 
is contained in a finite number of worms.

The upper bound obtained in the proof of Theorem 
\ref{thm:infinite_orientations} is not optimal. For instance,
the angle changes induced by turns are not always independent:
the pair of turns induced by a tile of type 1 (right-handed, say) 
followed by a turn induced by a tile of type 2 (right-handed)
yields the same relative orientation as 
the pair of turns induced by a tile of type 2 (right-handed) 
followed by a turn induced by a tile of type 1 (right-handed).
However, we don't see how to bring this upper bound below
$O\big( (2m)^k \big)$. 

Throughout the paper we made the general assumption that all tilings 
are vertex-to-vertex, and that all tilings are locally finite. The 
requirement of local finiteness is essential, otherwise the Travel 
Lemma would become false: We can no longer guarantee that any two
tiles can be connected by a finite sequence of tiles at all.

The requirement that the tilings are vertex-to-vertex allows for the 
definition of worms. However, if the tilings are locally finite,
but not vertex-to-vertex, then we can dissect each parallelogram into 
finitely many parallelograms such that the new tiling is
vertex-to-vertex. Thus all results apply also to locally
finite parallelogram tilings which are not vertex-to-vertex. 

Equivalent results to those proved in this paper should apply to tilings of $\R^d$ ($d \ge 2$) by parallelotopes. In order to generalise Lemma~\ref{lem:travel}, however, we need to enclose some tile $T$ with a finite collection of tiles. This would require further objects. In addition to the pseudo-lines for worms, one must also include pseudo-planes (and pseudo-hyperplanes) made up of linked parallelotopes sharing a parallel edge in common. 

\providecommand{\bysame}{\leavevmode\hbox to3em{\hrulefill}\thinspace}
\providecommand{\MR}{\relax\ifhmode\unskip\space\fi MR }
\providecommand{\MRhref}[2]{%
  \href{http://www.ams.org/mathscinet-getitem?mr=#1}{#2}
}
\providecommand{\href}[2]{#2}

\end{document}